\documentclass{amsart}
\usepackage[pagebackref]{hyperref}
\usepackage{amssymb}

\newcommand{\Q}{{\mathbb Q}}

\newcommand{\Z}{{\mathbb Z}}
\newcommand{\R}{{\mathbb R}}

\newcommand{\p}{{\mathfrak p}}

\newcommand{\fQ}{{\mathfrak Q}}
\newcommand{\fq}{{\mathfrak q}}
\newcommand{\fP}{{\mathfrak P}}
\newcommand{\OO}{{\mathcal O}}
\newcommand{\gen}{{\text{gen}}}
\newcommand{\sign}{\operatorname{sign}}
\newcommand{\Gal}{\operatorname{Gal}}
\newcommand{\Spl}{\operatorname{Spl}}
\newcommand{\disc}{\operatorname{disc}}

\newcommand{\Cl}{\operatorname{Cl}}

\newcommand{\rsp}{\raisebox{0em}[2.7ex][1.3ex]{\rule{0em}{2ex} }}
\newcommand{\ds}{\displaystyle}

\newcommand{\eq}{\stackrel{2}{=}}
\newcommand{\corr}{\stackrel{\text{gal}}{\longleftrightarrow}}

\newcommand{\too}{\longmapsto}
\newcommand{\la}{\langle}
\newcommand{\ra}{\rangle}

\newtheorem{lemma}{Lemma}
\newtheorem{cor}{Corollary}
\newtheorem{prop}{Proposition}
\newtheorem{theorem}{Theorem}

\theoremstyle{definition}

\theoremstyle{remark}
\newtheorem{rem}{Remark}

\begin{document}

\title{Unramified Quaternion Extensions of Quadratic Number Fields}
\author{Franz Lemmermeyer}
\address{M\"orikeweg 1, 73489 Jagstzell, Germany}
\email{hb3@ix.urz.uni-heidelberg.de}

\maketitle

\section*{Introduction}

The first mathematician who studied quaternion extensions
($H_8$-extensions for short) was Dedekind \cite{Ded}; 
he gave $\Q(\sqrt{(2+\sqrt2)(3+\sqrt6)\,}\,)$ as an example.  
The question whether given quadratic or biquadratic number 
fields can be embedded in a quaternion extension was extensively 
studied by Rosenbl\"uth \cite{Ro}, Reichardt \cite{R}, 
Witt \cite{W}, and Damey and Martinet \cite{DM}; see Ledet
\cite{Led} and the surveys \cite{JY} and \cite{GSS} for more 
details. Later, Fujisaki \cite{Fuj}, Kiming 
\cite{K} and Vaughan \cite{V} gave simple constructions of 
$H_8$-extensions of $\Q$.

In \cite{BK}, Bachoc and Kwon studied 
$H_8$-extensions of cyclic cubic number fields from an 
arithmetic viewpoint. The corresponding problem for
certain sextic fields was dealt with by Jehanne \cite{Jeh} 
and Cassou-Nogu\`es and Jehanne \cite{CNJ96}.

Quaternion extensions of $\Q$ also played a central role in the
theory of the Galois module structure of the ring of integers of
algebraic number fields (see Martinet's papers \cite{Mar,Mar2}), and the
Introduction of \cite{Fr} for a detailed account).  As Cohn
\cite{Cohn2} showed, quaternion extensions can also be used to 
explain congruences between certain binary quadratic forms.

Since quaternion extensions of $\Q$ always ramify over
their biquadratic subfield (see Cor. \ref{C1} below), they
do not occur as Hilbert class fields of quadratic or
biquadratic number fields. In order to find unramified
$H_8$-extensions one has to look at base fields $\ne \Q$. 
Already Furtw\"angler \cite{Fw} knew that such extensions exist,
Koch \cite{Koc} gave explicit criteria that guarantee their
existence, and it was Kisilevsky \cite{Kis} who showed that the 
second Hilbert $2$-class field of e.g. $k = \Q(\sqrt{-30}\,)$ is an 
$H_8$-extension of $k$. Hettkamp \cite{Het} found criteria for the
existence of unramified $H_8$-extensions of certain real quadratic
number fields, and finally M.~Horie \cite{Hor} gave the first 
explicit example of such an extension. 
Recently, Louboutin and Okazaki \cite{Lo1,Lo2,LO} have computed
relative class numbers of quaternion CM-extensions of $\Q$ 
as well as of unramified quaternion extensions of real quadratic
number fields (\cite{LO2}).

In \cite{Le1,Le2} we have shown how to construct unramified
quaternion extensions of a number field $k$ which is a
quadratic extension of a field $F$, and $F$ is totally real 
and has odd class number in the strict sense. In this
article we will show that this construction can be carried 
out with completely elementary methods (along the same lines
as the classical papers of R\'edei \cite{Red} and R\'edei \& 
Reichardt \cite{RR}) as long as we 
restrict ourselves to quadratic number fields. 

\section{Preliminaries}

We begin by introducing some notation. Let $k$ be a quadratic 
number field with discriminant $d$. An extensions $K/k$ is said
to be unramified if $\disc(K/k) = (1)$, i.e. if no finite
primes ramify. The genus class field 
$k_\gen$ of $k$ is defined as the maximal unramified
extension of $k$ which is abelian over $\Q$. It is known
that $k_\gen$ is the compositum of all unramified
quadratic extensions of $k$, and that $k_\gen = 
\Q(\sqrt{d_1}, \ldots, \sqrt{d_t}\,)$, where 
$d = d_1\cdots d_t$ is the factorization of $d$ into prime
discriminants. Since we intend this article to be as 
self contained as possible, we will give a proof for the part
of genus theory we need. First, however, we recall some basic
facts from Hilbert's theory of ramification in Galois extensions
(see, for example, \cite{Cohn} for proofs). Let $K/F$ be a finite
normal extension of number fields and put $G = \Gal(K/F)$. 
Moreover, let $\fP$ be a prime ideal in $\OO_K$. Then the 
stabiliser of $\fP$, 
$$Z_\fP(K/F) = \{\sigma \in G: \fP^\sigma = \fP\},$$ 
is called the {\em decomposition group}, and 
its subgroup 
$$T_\fP(K/F) = \{\sigma \in Z_\fP(K/F): 
     \alpha^\sigma \equiv \alpha \bmod \fP 
	\text{ for all } \alpha \in \OO_K\}$$
the {\em inertia group} of $\fP$. The order of $T_\fP(K/F)$ is equal 
to the ramification index of $\fP$ in $K/F$. The residue class 
field $\OO_K/\fP$ is a finite extension of $\OO_F/\p$ with 
Galois group isomorphic to $Z_\fP/T_\fP$; in particular, 
$T_\fP$ is a {\em normal} subgroup of $Z_\fP$.
We will make use of the following properties of the Hilbert
sequence:

\begin{prop}\label{P0}
Let $K/F$ be a finite normal extension of number fields,
and let $\fP$ be a prime ideal in $\OO_K$ lying over the prime
ideal $\p$ in $\OO_F$. 
\begin{enumerate}
\item $\p$ splits completely in a subextension $k/F$ 
	($\p \in \Spl(k/F)$) if and 
	only if $Z_\p \subseteq \Gal(K/k)$;
\item Let $\p_k = \fP \cap \OO_k$ be the prime ideal in $\OO_k$
	lying below $\fP$; then $\p$ is unramified in $k/F$ 
	if and only if $T_\fP(K/k) \subseteq \Gal(K/k)$;
\item Let $T(K/F)$ be the subgroup of $\Gal(K/F)$ generated
	by all the inertia subgroups $T_\p(K/F)$; then the
	fixed field $k$ of $T(K/F)$ is the maximal unramified
	extension of $F$ contained in $K$, and $k/F$ is normal.
\end{enumerate}
\end{prop}

\begin{cor}\label{ChM} (Chebotarev's Monodromy Theorem)
Let $k$ be a quadratic number field, and suppose that 
$K/k$ is unramified and that $K/\Q$ is normal. Then
the Galois group of $K/\Q$ is generated by elements
of order $2$ not contained in $\Gal(K/k)$.
\end{cor}

\begin{proof}
Since $\Q$ does not have nontrivial unramified extensions,
the group $T$ generated by all inertia subgroups must fix
$\Q$ (by Prop. \ref{P0}, part. 3.); this shows that 
$T = \Gal(K/\Q)$. Since $K/k$ is unramified, we have
$T \cap \Gal(K/K) = \{1\}$.  
\end{proof}

\begin{cor}\label{C1}
If $K/\Q$ is an $H_8$-extension of $\Q$, then there exists a prime
with ramification index $4$. In particular, $K$ is ramified
over every quadratic subfield of $K/\Q$. 
\end{cor}

\begin{proof}
Since $H_8$ cannot be generated by elements of order $2$, there
must be some inertia group $T_\fP$ of order divisible by $4$. 
Then $\fP$ is completely ramified over its (quadratic) inertia
subfield; in particular, $\fP$ ramifies in $K/k$, where $k$
is the biquadratic subfield of $K/\Q$. 
\end{proof}

\begin{cor}\label{C2}
Let $K/F$ be a cyclic quartic extension with quadratic
subfield $k$. Then every prime ramifying in $k/F$ also
ramifies in $K/k$. 
\end{cor}

\begin{proof}
Since $\p$ ramifies in $k/F$ and $K/F$ is cyclic of 
prime power degree, $F$ must be the inertia subfield of 
$\p$ in $K/F$. 
\end{proof}

The following result is well known (\cite{Cohn}, 14.33):

\begin{prop}
Let $k$ be a quadratic number field with discriminant $d$. Then
$d$ can be written uniquely as a product of prime discriminants.
\end{prop}

Finally, the next proposition contains all the genus theory
we will need:

\begin{prop}\label{P3}
Let $k$ be a quadratic number field and suppose that $K/k$
is an unramified quadratic extension. Then $K/\Q$ is normal,
and $\Gal(K/k) \simeq (2,2)$.
\end{prop}

\begin{proof}
If $K/\Q$ is not normal, let $\sigma$ be the nontrivial
automorphism of $k/\Q$. Then $N = KK^\sigma$ is the normal
closure of $K/\Q$, and $\Gal(N/\Q)$ is a nonabelian group of
order $8$ with a subgroup $\Gal(N/k)$ of type $(2,2)$.
The only such group is the dihedral group of order $8$.
Since $K/k$ is unramified, so are $K^\sigma/k$ and $N/k$. 
Let $F$ be the quadratic subfield of $N/\Q$ over which
$N$ is cyclic, and let $M$ be its quadratic subextension. 
Since $N/M$ is unramified, so is $M/F$ by Cor. \ref{C2}.
But then $M$ is unramified over $k$ and $F$.
Since $M/\Q$ is bicyclic, it contains three quadratic
subfields, $k$, $F$, and $\widetilde{k}$, say.
Let $p$ be any prime ramified in $\widetilde{k}/\Q$.
Since $M/k$ is unramified, $p$ has inertia degree $2$ in
$M/\Q$, hence its inertia subfield in $M$ is $k$ or $F$. 
But this contradicts the fact that $M/k$ and $M/F$ are unramified.
\end{proof}

\section{Construction of $H_8$-Extensions}
Now let $M/k$ be an unramified normal extension of $k$ with 
Galois group 
$$\Gal(M/k) = H_8 = \la \sigma, \tau: \sigma^2 = 
	\tau^2 = -1, \sigma\tau = -\tau\sigma\ra,$$ 
the quaternion group of order $8$ (observe that $-1$ denotes a central 
involution, i.e. an automorphism of order $2$ contained in the center 
of the group). We will also assume that $M$ is normal over $\Q$.

Our first claim is that $\Gamma = \Gal(M/\Q) \simeq D_4 \curlyvee C_4$.
In fact, let $K$ be the quartic subextension of $M/k$; then
$K/k$ is an elementary abelian unramified extension of $k$,
hence contained in the genus class field of $k$.
In particular we see that $\Gal(K/\Q) \simeq (2,2,2)$; therefore
$\Gamma$ is a group of order $16$ with a subgroup of type
$H_8$ and a factor group of type $(2,2,2)$. There are only
two such groups (see \cite{HS} or \cite{TW}), 
i.e. $C_2 \times H_8$ and $D_4 \curlyvee C_4$.
But $C_2 \times H_8$ cannot be generated by elements
of order $2$, therefore we must have  $\Gamma \simeq D_4 \curlyvee C_4$.

Now we put 
$\Gamma = \la \rho, \sigma, \tau:
	\rho^2 = \sigma^2 = \tau^2 = -1, [\rho,\sigma] = [\rho,\tau] = 1,
	\sigma\tau = -\tau\sigma\ra.$
$\Gamma$ has seven subgroups of order $8$; three of them
(those containing $\rho$) have type $(2,4)$, three are
dihedral groups (namely $\Delta_1 = \la \sigma, \rho\tau\ra$,
$\Delta_2 = \la \tau, \rho\sigma \ra$, and 
$\Delta_3 = \la \rho\sigma, \rho\tau\ra$), and one of them is
the quaternion subgroup $\la \sigma, \tau\ra$ fixing the
field $k$ (see Table 1). 

\begin{table}[hb]
\caption{}
\begin{center} 
\begin{tabular}{|c|ccl|c|} \hline
\rsp $G$  & & & fixed field of  $G$ \quad  & $G \simeq $ \\ \hline
\rsp $\la -1, \rho\sigma\tau \ra$ & \quad $K_{12} $ & $=$ & 
	$\Q(\sqrt{d_1}, \sqrt{d_2}\,)$ \quad  & $(2,2)$ \\
\rsp $\la -1, \rho\sigma \ra$ & \quad $K_{23} $ & $=$ & 
	$\Q(\sqrt{d_2}, \sqrt{d_3}\,)$ \quad  & $(2,2)$ \\
\rsp $\la -1, \rho\tau \ra$ 	& \quad $K_{13} $ & $=$ & 
	$\Q(\sqrt{d_1}, \sqrt{d_3}\,)$ \quad  & $(2,2)$ \\
\rsp $ \la \rho \ra$	& \quad $L_0 $ & $=$ & 
	$\Q(\sqrt{d_1d_2}, \sqrt{d_1d_3}\,)$ \quad & $(4)$ \\
\rsp $\la \sigma \ra$ & \quad $L_1 $ & $=$ & 
	$\Q(\sqrt{d_1}, \sqrt{d_2d_3}\,)$ \quad & $(4)$ \\
\rsp $\la \tau \ra$ & \quad $L_2 $ & $=$ & 
	$\Q(\sqrt{d_2}, \sqrt{d_1d_3}\,)$ \quad & $(4)$ \\
\rsp $\la \sigma\tau \ra$ & \quad $L_3 $ & $=$ & 
	$\Q(\sqrt{d_3}, \sqrt{d_1d_2}\,)$ \quad & $(4)$ \\
	  \hline
\rsp $ \la \sigma, \tau \ra$ & \quad $k $ & $=$ & 
	$\Q(\sqrt{d}\,)$  & $H_8$ \\
\rsp $\la \rho, \sigma\tau \ra$ & \quad  $k_{12} $ & $=$ & 
	$\Q(\sqrt{d_1d_2}\,)$     \quad  & $(2,4)$ \\
\rsp $\la \rho, \sigma \ra$ & \quad  $k_{23} $ & $=$ & 
	$\Q(\sqrt{d_2d_3}\,)$ 	\quad  & $(2,4)$ \\
\rsp $\la \rho, \tau \ra$ & \quad  $k_{13} $ & $=$ & 
	$\Q(\sqrt{d_1d_3}\,)$ 	\quad  & $(2,4)$ \\
\rsp $\la \sigma,\rho\tau \ra$ &\quad  
	$k_1 $ & $=$ & $\Q(\sqrt{d_1}\,)$ 	\quad  & $D_4$ \\
\rsp $\la \tau,\rho\sigma \ra$ &\quad  
	$k_2 $ & $=$ & $\Q(\sqrt{d_2}\,)$ 	\quad  & $D_4$ \\
\rsp $\la \rho\sigma,\rho\tau \ra$ & \quad  
	$k_3 $ & $=$ & $\Q(\sqrt{d_3}\,)$ 	\quad & $D_4$ \\ 
		\hline 
\rsp $\la \rho, \sigma, \tau \ra $ 
		&  \quad  & & $\Q$  \quad & $D_4 \curlyvee C_4$ \\
\hline
\end{tabular}
\end{center}
\end{table}

The fixed field of $\Delta_j$ ($1 \le j \le 3$) is a quadratic
number field $k_j$ with discriminant $d_j$. We claim that the
$d_j$ are relatively prime. In fact, assume that $\p$
is a prime ideal which ramifies in at least two of
the three fields $k_j = \Q(\sqrt{d_j}\,)$, say in $k_1$ and $k_2$; 
let $T_\p(M/\Q)$ denote the inertia subgroup of $\Gamma$. Then 
$T_\p(M/\Q)$ has the following properties:
\begin{enumerate}
\item  $T_\p(M/\Q)$ has order $2$: this follows from $M/k$
	being unramified;
\item $T_\p(M/\Q) \cap \Delta_1 = T_\p(M/\Q) \cap \Delta_2 = \{ 1 \}$,
	because $\p$ cannot ramify in $M/k_j$ ($1 \le j \le 2$).
\end{enumerate}
But now we see that $\Delta_1 \cup \Delta_2$ contain all seven elements
of order $2$, hence at least one of them must contain the
element of order $2$ which generates $T_\p(M/\Q)$.
The same argument applied to an infinite prime yields that
at most one of the discriminants $d_j$ can be negative.

Now $K$ contains the three quadratic subfields $k_j$; from
degree considerations it is clear that we must have
$K = \Q(\sqrt{d_1}, \sqrt{d_2}, \sqrt{d_3}\,)$. 
In particular, $K$ contains the quadratic number field with 
discriminant $d' = d_1d_2d_3$; we claim that $d = d'$. Since 
$K \subseteq k_\gen$, we see that $d_1d_2d_3$ divides $d$ 
(otherwise $K/k$ would ramify). On the other hand,
$\Q(\sqrt{d}\,) \subset K$ shows that 
$d$ is the product of some of the $d_i$; therefore we must
have equality $d = d_1d_2d_3$.

So far we have seen: if there is an unramified quaternion extension
$M/k$, then $d = \disc k = d_1d_2d_3$ is the product of three
relatively prime discriminants $d_j$, at most one of which is
negative. 

Next we will study the decomposition of primes in $M/\Q$. 
To this end, consider a prime $p\mid d_1$; then $p \OO_1= \p^2$
for some prime ideal $\p$ in the ring of integers $\OO_1$ of
$k_1$. Let $F$ be the fixed field of $\la \sigma \ra$; then
$F$ contains the fixed field $k$ of $\la \sigma, \tau \ra$, 
the fixed field $k_1$ of $\la \sigma, \rho\tau \ra $ and
the fixed field $k_{23} = \Q(\sqrt{d_2d_3}\,)$ of 
$ \la \sigma, \rho \ra$. We claim that $p$ splits in $k_{23}/\Q$.
We already know that the inertia subgroup $T = T_p(M/\Q)$ has order
$2$; since the prime ideal $\p$ above $p$ in $k_1$ does
not ramify in $M/k_1$, we must have $T \cap \Delta_1 = \{ 1 \}$. 
Enumerating the subgroups of order $2$ in $\Gamma$ shows that 
there are only the possibilities $T = \la \rho\sigma \ra$
and $T = \la -\rho\sigma \ra$. But now the normaliser of $T$
in $\Gamma$ equals $N_\Gamma(T) = \la \rho, \sigma\ra$. Since
$T$ is a normal subgroup of the decomposition group $Z = Z_p(M/\Q)$
of $p$, we conclude that $Z \subseteq \la \rho, \sigma\ra$. 
But this means that the fixed field $k_{23}$ of $\la \rho, \sigma\ra$
is contained in the decomposition field for $p$, i.e. 
$p$ splits in $k_{23}/\Q$, and we have $(d_2d_3/p_1) = +1$
for all primes $p_1$ dividing $d_1$. By symmetry 
we conclude

\begin{prop}
Let $k$ be a quadratic number field with discriminant
$d$. If there exists an unramified extension $M/k$ with
$\Gal(M/k) \simeq H_8$ and which is normal over $\Q$,
then 
\begin{enumerate}
\item $\Gal(M/\Q) \simeq D_4 \curlyvee C_4$;
\item there is a factorization $d = d_1d_2d_3$ of $d$ into
	three discriminants $d_1, d_2, d_3$, at most one of 
	which is negative;
\item for all primes $p_i \mid d_i$ we have
	$ (d_1d_2/p_3) = (d_2d_3/p_1) = (d_3d_1/p_2) = +1.$ 
\end{enumerate}
\end{prop}

\begin{rem}
We will call a factorization $d = d_1d_2d_3$ of $d$
into discriminants an $H_8$-factorization, if the condition
$ (d_1d_2/p_3) = (d_2d_3/p_1) = (d_3d_1/p_2) = +1$ is
satisfied for all $p_i\mid d_i$. It is an easy exercise to show
that the quadratic reciprocity law implies that at most
one of the $d_i$ is negative.
\end{rem}

Our next task is the explicit construction of the
unramified $H_8$-extension $M/k$. To this end, assume that we
have already found it, and let $\Gamma = \Gal(M/\Q)$ be
as above. Then $K_{13} = \Q(\sqrt{d_1}, \sqrt{d_3}\,)$ is the fixed
field of $\la -1, \rho \tau \ra$ (we will write 
$K_{13} \corr \la -1, \rho \tau \ra$); therefore $M/K_{13}$ is an
extension of type $(2,2)$ with subfields
$K_{13}(\sqrt{d_2}\,) \corr \la -1 \ra$, 
$N = K_{13}(\sqrt{\mu}\,) \corr \la \rho\tau \ra$ and 
$N' = K_{13}(\sqrt{\nu}\,) \corr \la -\rho\tau \ra$ for 
some $\mu, \nu \in K_{13}$. Now $\sigma$ and $\rho$ act on 
$\la -1, \rho \tau \ra$, $\sigma$ is trivial on $\Q(\sqrt{d_1}\,)$
and $\rho$ on $\Q(\sqrt{d_3}\,)$. Since $(\rho\tau)^\sigma
= -\rho\tau$ and $(\rho\tau)^\rho = \rho\tau$, we see that
$N^\sigma = N'$, $N^\rho = N$ and $N^{\rho\sigma} = N'$.
In particular we can choose $\nu = \mu^\sigma$.

Of course, any $\mu \in L$ can be written in the form
$\mu = x + y\sqrt{d_1} + z\sqrt{d_3} + w\sqrt{d_1d_3}$. 
But for the construction of the class fields it would be 
preferable if $ ´\mu$ could be factorized into elements 
coming from subfields of $K_{13}$. Let us therefore put
$\mu = (x_1+y_1\sqrt{d_1}\,)(x_3+y_3\sqrt{d_3}\,)
	(x\sqrt{d_1}+y\sqrt{d_3}\,)$, where the coefficients
are rational. Since $K_{13}(\sqrt{\mu\mu^\sigma}\,) = K_{13}(\sqrt{d_2}\,)$,
we conclude that  $\mu\mu^\sigma \eq d_2$ (the symbol $\eq$ 
indicates that the two sides differ only by a square in $K_{13}$).
We find:
\begin{align*} 
\mu^{1+\sigma} &= 
	(x_1+y_1\sqrt{d_1}\,)^2(x_3^2-d_3y_3^2)(d_1x^2-d_3y^2)
	\eq d_2 \\
\mu^{1+\rho} &= 
	(x_1^2-d_1y_1^2)(x_3^2-d_3y_3^2)(x\sqrt{d_1} + y\sqrt{d_3}\,)^2
	\eq -1 \\
\mu^{1+\rho\sigma} &= 
	(x_1^2-d_1y_1^2)(x_3+y_3\sqrt{d_3}\,)^2(-d_1x^2+d_3y^2) 
	\eq -d_2 
\end{align*}
If we put $a:= x_1^2-d_1y_1^2$, then the second equation yields
$-a \eq x_3^2-d_3y_3^2$, and from the first equation we get
$d_1x^2 -d_3y^2 \eq -ad_2$. This suggests that in order to
construct $D_4 \curlyvee C_4$-extensions of $\Q$ we should 
try to solve the following system of diophantine equations over $\Z$:

\begin{align*}
d_1X_1^2 - d_2X_2^2 & = -ad_3X_3^2  & (I)   \\
 Y_1^2 - d_1 Y_2^2  & = aY_3^2 	    &  (II) \\
 Z_1^2 - d_2 Z_2^2  & = -aZ_3^2     &  (III)
\end{align*} 

The same system of equations was given by 
M.~Horie \cite{Hor}; moreover, the construction 
presented by Min\'ac and Smith \cite{MS} 
uses three equations which can easily be shown to be 
equivalent to (I)-(III). In the form of the
"common slot property" (cf. T.~Smith \cite{Sm}, Prop. 1.1.2)
of quaternion algebras it seems to have been well known to  
people familiar with Brauer groups. Now we claim

\begin{prop}\label{Pcon}
Let $k$ be a quadratic number field with discriminant $d$, 
and suppose that $d = d_1d_2d_3$ is an $H_8$-factorization.
Then there exists an odd squarefree $a \in \Z$ such that
the system (I) -- (III) of diophantine equations
has nontrivial solutions in $\Z$. If $x_i, y_i, z_i \in \Z$ form
a solution, put
$$ \mu \ = \ (x_1\sqrt{d_1}+x_2\sqrt{d_2}\,)
		(y_1+y_2\sqrt{d_1}\,)(z_1+z_2\sqrt{d_2}\,)/r,$$
where $r \in \Z$ is an arbitrary nonzero integer. Then
$M = \Q(\sqrt{d_1}, \sqrt{d_2}, \sqrt{d_3}, \sqrt{\mu}\,)$
is an $H_8$-extension of $k$ which is normal over $\Q$ with
$\Gal(M/\Q) \simeq D_4 \curlyvee C_4$. 
If we choose $r \in \Z$ in such a way that $\mu$ is integral 
and not divisible by any rational prime $p$, then there is a 
$2$-primary element in $\{\pm \mu \}$ if $d_1d_2 \equiv 0, 1 \bmod 8$,
and in $\{\mu, 2\mu\}$ if $d_1d_2 \equiv 4 \bmod 8$. 
\end{prop}

\subsection*{Existence of $a$}
We want to show that the diophantine equations (I)--(III) 
have nontrivial solutions if we choose $a \in \Z$ suitably. 
To this end we write $d_1 = \prod d_{1,i}$ and 
$d_2 = \prod d_{2,i}$ as products of prime discriminants.
We will assume without loss of generality that $d_2 > 0$ 
(otherwise we simply exchange $d_1$ and $d_2$ -- recall that
at most one of the $d_i$ is negative).
Then we claim that (I)--(III) are nontrivially solvable 
in $\Z$ if and only if the following conditions are satisfied:
\begin{enumerate}
\item[(1)] $a > 0$ if $d_1 < 0$;
\item[(2)] $(d_1/a) = (d_2/a) = +1$;
\item[(3)] $(d_{1,i}/a) = +1$ for all discriminants 
	$d_{1,i} \mid d_1$;
\item[(4)] $(d_{2,i}/a) = \sign(d_{2,i})$ for all discriminants 
	$d_{2,i} \mid d_2$.
\end{enumerate}

Before we prove this, let us show that such $a \in \Z$ 
actually exist. If $d_2$ is a sum of two squares, then we can 
obviously choose $a = 1$.  If not, then we can find an odd prime 
$a$ satisfying properties (1), (3) and (4) 
by making use of quadratic reciprocity and Dirichlet's theorem 
on primes in arithmetic progressions (here we use that $d_1$ 
and $d_2$ are relatively prime). We claim that all such primes 
satisfy condition (2) automatically. In fact this is obvious
for the condition on $d_1$; since $d_2$ was assumed to be 
positive, the number of negative $d_{2,i}$ is even, and this 
implies $1 = \prod (d_{2,i}/a) = (d_2/a)$.

Now consider the diophantine equation (I): $d_1X_1^2 - d_2X_2^2 =
-ad_3X_3^2$. We have to show that it has solutions in 
the reals, modulo $a$ and modulo every odd prime dividing $d$
(we can neglect the prime $2$ because of the product formula). 
Solvability in $\R$ is clear: if (II) is solvable in $\R$
then $d_1 <0$ clearly implies $a > 0$. Moreover, this 
condition suffices for the solvability in $\R$ of (I) and (II);
since $d_2 > 0$, (III) always has real solutions. 
If we reduce (I) modulo $a$, then we get
$d_1X_1^2 \equiv d_2X_2^2$; since $(d_1/a) = (d_2/a) = +1$, 
this equation is indeed solvable.
Now let $p$ be an odd prime dividing $d_1$; we get
$d_2X_2^2 \equiv ad_3X_3^2 \bmod p$. The condition
$(d_2/p) = (d_3/p)$ shows that solvability is equivalent to
$(a/p) = +1$. Quadratic reciprocity shows that this is
equivalent to $(p^*/a) = +1$, where $p^* = (-1)^{(p-1)/2}p$
is the unique prime discriminant divisible by $p$;
now condition (3) guarantees solvability.
Next assume that $p \mid d_2$; then we have to solve 
$d_1X_1^2 \equiv -ad_3X_3^2 \bmod p$. Again, the condition
$(d_1/p) = (d_3/p)$ reduces this to a proof of $(-a/p) = +1$.
If $p \equiv 1 \bmod 4$, this is equivalent to $(p/a) = +1$,
which holds because of (4). If  $p \equiv 3 \bmod 4$,
then $(-a/p) = +1 \iff (a/p) = -1 \iff (p^*/a) = -1$, and
again this is true by (4). 
Finally let $p \mid d_3$; then $d_1X_1^2 \equiv d_2X_2^2  \bmod p$
is clearly solvable.
The equations (II) and (III) can be treated similarly.

\subsection*{Computation of the Galois Group}

Let $K= \Q(\sqrt{d_1}, \sqrt{d_2}, \sqrt{d_3}\,)$; then $K/k$
is unramified and we have $\Gal(K/k) \simeq (2,2)$. If
$\alpha, \beta \in K^\times$ satisfy an equation $\alpha = \beta \xi^2$
for some $\xi \in K^\times$, then we write $\alpha \eq \beta$.
For the computation of $\Gal(M/k)$ we need a result which
was stated without proof by Furtw\"angler \cite{Fw}:

\begin{lemma}\label{L1}
Let $K/F$ be a quartic extension with $\Gal(K/F) \simeq (2,2)$;
let $\sigma, \tau$ and $\sigma\tau$ denote its
nontrivial automorphisms, and put $M=K(\sqrt{\mu}\,)$. 
Then $M/F$ is normal if and only if $\mu^{1-\rho} \eq 1$ 
for all $\rho \in \Gal(K/F)$. If this is the case, write
$\mu^{1-\sigma} = \alpha_\sigma^2$, 
$\mu^{1-\tau} = \alpha_\tau^2$ and
$\mu^{1-{\sigma\tau}} = \alpha_{\sigma\tau}^2$. It is easy 
to see that $\alpha_\rho^{1+\rho} = \pm 1$ for all 
$\rho \in \Gal(K/F)$; define 
$S(\mu,K/F) = (\alpha_\sigma^{1+\sigma}, \alpha_\tau^{1+\tau}, 
	\alpha_{\sigma\tau}^{1+{\sigma\tau}})$ and identify 
vectors which coincide upon permutation of their entries. Then
$$\Gal(M/F) \simeq
\begin{cases}
(2,2,2) & \iff S(\mu,K/F) = (+1, +1, +1),  \\
(2,4)   & \iff S(\mu,K/F) = (-1, -1, +1), \\
 D_4    & \iff S(\mu,K/F) = (-1, +1, +1), \\
 H_8    & \iff S(\mu,K/F) = (-1, -1, -1).
\end{cases}$$
Moreover, $M$ is cyclic over the fixed field of $\la \rho \ra$
if and only if $\alpha_\rho^{1+\rho} = -1$, and has type $(2,2)$
otherwise.
\end{lemma}

\begin{proof}
Let $K/k$ be a quadratic extension and put $M = K(\sqrt{\mu}\,)$
for some $\mu \in K$. Let $\sigma$ denote the nontrivial
automorphism of $K/k$. Then $M/k$ is normal if and only
if $M^\sigma = M$, and by Kummer Theory this is equivalent
to $\mu^\sigma \eq \mu$, i.e. to $\mu^{1-\sigma} = \alpha_\sigma^2$
for some $\alpha_\sigma \in K^\times$. Since 
$(\alpha_\sigma^2)^{1+\sigma} = \mu^{(1-\sigma)(1+\sigma)} = 1$,
we see that $\alpha_\sigma = \pm 1$.

Next suppose that $M/k$ is normal; then 
$\widetilde{\sigma}: a+b\sqrt{\mu} \too 
	a^\sigma + b^\sigma \alpha_\sigma \sqrt{\mu}$
defines an automorphism of $M/k$ whose restriction to
$K/k$ coincides with $\sigma$. But now
$\widetilde{\sigma}^2: a+b\sqrt{\mu} \too 
      a+b\alpha^{1+\sigma} \sqrt{\mu}$,
hence $\widetilde{\sigma}$ has order $4$ if $\alpha^{1+\sigma}  = -1$
and order $2$ if $\alpha^{1+\sigma}  = +1$.

Now clearly $M/F$ will be normal if and only if $\mu^\rho \eq \mu$
for all $\rho \in \Gal(K/F)$, i.e. if and only if $M/k_i$ is
normal for all three quadratic subextensions $k_i$ of $K/k$.
Moreover, the noncyclic groups of order $8$ can be classified 
by their number of automorphisms of order 4: this number is
$0, 1, 2$ or $3$ if $G \simeq (2,2,2), D_4, (2,4)$ or $H_8$,
respectively. The claims of Lemma \ref{L1} now follow.
\end{proof}

Lemma \ref{L1} reduces the verification of $\Gal(M/k) \simeq H_8$
to a simple computation; in order to simplify the notation, put 
$\beta = (x_1\sqrt{d_1}+x_2\sqrt{d_2}\,)$,
$\gamma = (y_1+y_2\sqrt{d_1}\,)$ and 
$\delta = (z_1+z_2\sqrt{d_2}\,)$; then $\mu = r\beta\gamma\delta$,
and we find 
\begin{align*}
\alpha_\sigma 	& = ax_3z_3\sqrt{d_3}/(\beta^\sigma \delta^\sigma),
	      	& \alpha_\sigma^{1+\sigma} 	& = -1 \\
\alpha_\tau	& = ax_3y_3\sqrt{d_2}/(\beta^\sigma\gamma^\tau),
		& \alpha_\tau^{1+\tau} 		& = -1 \\
\alpha_{\sigma\tau} & = ay_3z_3/(\gamma^\tau\delta^\sigma),
		& \alpha_{\sigma\tau}^{1+{\sigma\tau}} & = -1.
\end{align*}
Therefore $\Gal(M/k) \simeq H_8$. Next we check that $M/\Q$
is normal. To this end, let $\rho$ be an extension of the
nontrivial automorphism of $k/\Q$, say
$\rho:\sqrt{d_i} \longmapsto -\sqrt{d_i}$ for $1 \le i \le 3$. 
Then $M/\Q$ is normal if and only of $M^\rho = M$, i.e. 
if $\mu^{1-\rho} \eq 1 \bmod K^\times$. A small computation
shows that $\mu^{1-\rho} = (ay_3z_3/(\gamma^\tau \delta^\sigma))^2$
and $\alpha_\rho^{1+\rho} = -1$. 

Finally we verify that $M/\Q(\sqrt{d_i}\,)$ is a $D_4$-extension
for $1 \le i \le 3$ (by using Lemma \ref{L1} again); this implies 
that $\Gal(M/\Q)$ is not isomorphic to $C_2 \times H_8$, and now 
$\Gal(M/\Q) \simeq D_4 \curlyvee C_4$ follows.

\subsection*{Ramification outside $2\infty$}
We start with $r = 1$ and integral solutions $x_i, y_i, z_i$ of our
system of equations. Let $\p$ be a prime ideal in 
$K$ lying above an odd prime $p$.  

Suppose first that $p \nmid d$. If $\p$ ramifies 
in $M/K$, then $\p \mid \mu$; since $M/\Q$ is normal,
all conjugates of $\p$ also ramify, hence also divide $\mu$,
and we find that $p \mid \mu$. Replacing $r$ by $r/p$  
we see that we can find an $r \in \Q$ such that $\mu$
is integral and not divisible by any prime not dividing $2d$,
and that the corresponding $\mu$ defines a quaternion
extension $M/k$ which is unramified outside $2d\infty$. 

Now assume that $p\mid d$ is odd. Then $\p$ is unramified 
in $K/k$, so if $\p$ ramifies in $M/K$ then the ramification
index $e_p(M/\Q)$ must equal $4$. The inertia subfield $M_T$ 
of $p$ has therefore degree $4$ over $\Q$, and all such fields are
easily seen to be $V_4$-extensions of $\Q$. If we assume
(without loss of generality) that $p \mid d_3$, then since
$p$ does not ramify in its inertia field, we must have 
$M_T = \Q(\sqrt{d_1}, \sqrt{d_2}\,)$. But now $M/M_T$
is a $V_4$-extension, so that only primes above $2$ can ramify
completely. This shows that all odd $p\mid d$ are
unramified in $M/k$. 

\subsection*{Ramification above $2$}
We will assume that 
$$\mu = r(x_1\sqrt{d_1} + x_2 \sqrt{d_2}\,)
	(y_1+y_2\sqrt{d_1}\,)(z_1+z_2\sqrt{d_2}\,) 
		\in K_{12} = \Q(\sqrt{d_1},\sqrt{d_2}\,)$$
is integral and not divisible by any rational prime $p$; 
moreover, we suppose that $d_1d_2$ is odd, i.e. that $2$
is unramified in $K_{12}/\Q$. We have to show that $\mu$ 
or $-\mu$ is $2$-primary in $K$, i.e. that $2$ does not 
ramify in at least one of the extensions $K(\sqrt{\mu}\,)/K$ or 
$K(\sqrt{-\mu}\,)/K$. 

Since $2 \nmid \mu$ and $2$ is not ramified, there is a 
prime ideal $\fQ$ above $2\OO_K$ which does not divide 
$\mu$. If we can show that $\mu$ is $\fQ$-primary, then 
$\fQ$ does not ramify in $M/K$; but $M/\Q$ is normal, so 
if $\fQ$ does not ramify, neither does any other
prime ideal above $2$, and our claim follows. It is
obviously sufficient to show that $\pm \mu$ is 
$\fQ$-primary in $K_{12}$. 

Suppose that $2$ splits completely
in $K_{12}/\Q$. Then $\OO_{12}/\fQ^m \simeq \Z/2^m\Z$,
hence $\mu \equiv a \bmod \fQ^2$ for some odd
integer $a$. But either $a$ or $-a$ is congruent to $1 \bmod 4$,
hence $\pm \mu \equiv 1 \bmod \fQ^2$ is $2$-primary.

Next assume that $\fQ$ has inertia degree $2$ in
$K_{12}/\Q$. Since the norm of $\mu$ to the three
quadratic subfields is either a square or $d_3$ times
a square, these norms are $2$-primary. Now we use the following

\begin{lemma}\label{L2}
Let $K/k$ be a quadratic extension, and let $\fq$ be a prime 
ideal in $\OO_k$ above $2$; assume moreover that $\fq$ is inert 
in $K/k$. If $N_{K/k} \mu$ is $\fq$-primary in $\OO_k$ and 
$\fq \nmid \mu$ then there exists an 
$\alpha \in \OO_k \setminus \fq$ such that 
$\mu \alpha$ is $\fq$-primary.
\end{lemma}

Assuming the truth of the lemma for the moment,
we find that it is sufficient to show that
$\alpha \in k_i$ is $\fq$-primary. But since $2$
splits in $k_i/\Q$, we have $\alpha \equiv \pm 1
\bmod \fq^{2}$, and we are done.

\begin{proof}[Proof of Lemma \ref{L2}]
We first claim that there exists a $\xi \in \OO_k 
\setminus \fq$ such that $\mu\xi^2 + \mu'{\xi'}^2 
\not\equiv 0 \bmod \fq$, where $\mu'$ denotes the
conjugate of $\mu$ with respect to $K/k$.
In fact if $\fq \nmid (\mu + \mu')$ then we can take 
$\xi = 1$; assume therefore that $\fq \mid (\mu + \mu')$. Since 
the trace is surjective in extensions of finite fields, there 
exists a $\xi \in \OO_K$ such that 
$\xi + \xi' \equiv 1 \bmod \fq$.  From $\fq \mid 2$ we get  
$\mu\xi^2 + \mu'{\xi'}^2 \equiv \mu(\xi+\xi')^2 \bmod \fq.$
Put $\nu = \mu \xi^2$; then $K(\sqrt{\mu}\,) = K(\sqrt{\nu}\,)$.
Since $\mu\mu' \equiv \eta_0^2 \bmod \fq^2$ for some 
$\eta_0 \in \OO_k$, we find
$\nu \nu' = \mu\mu'(\xi\xi')^2 \equiv  \eta^2 \bmod \fq^{2}$ 
for $\eta = \eta_0(\xi\xi') \in \OO_k$.
This implies at once that $\sqrt{\nu \nu'} \equiv \eta \bmod \fq$
in $\OO_L$, where $L = k(\sqrt{\mu\mu'}\,)$. 
Put $\alpha = \nu + \nu' + 2\eta$;
then $\alpha \in \OO_k$, $\fq\nmid \alpha$, and
$\alpha\nu \equiv \nu(\nu + \nu' + 2\sqrt{\nu\nu'}\,) 
	= (\nu+\sqrt{\nu\nu'}\,)^2 \bmod \fq^{2}$.
Therefore $\mu \alpha$ is $\fq$-primary in $KL$; but since
$L/k$ is unramified above $\fq$, this implies that 
$\mu \alpha$ is $\fq$-primary in $K$. 
\end{proof}

It remains to prove our claims if $2$ ramifies in $K_{12}/\Q$.
To this end we need

\begin{lemma}\label{L3}
Let $k$ be a quadratic number field with discriminant 
$d = d_1d_2d_3$ and assume that $M = K(\sqrt{\mu}\,)$ 
and $N = K(\sqrt{\nu}\,)$ are two $H_8$-extensions of $k$,
both constructed as in Prop. \ref{Pcon}. In particular,
they have the common subfield 
$K = k(\sqrt{d_1}, \sqrt{d_2}, \sqrt{d_3}\,)$.
Then there exists a $\delta \in \Z$ such that
$\nu \eq \delta\mu$. If, moreover, $M/k$ and $N/k$
are unramified outside $\infty$, then $\delta$ can be
chosen to be a discriminant dividing $d$. 
\end{lemma}

\begin{proof}
Let $\Gal(K/\Q) = \la \rho, \sigma, \tau\ra \simeq (2,2,2)$.
If $M = N$, then $\nu \eq \mu$ and there is nothing to prove.
Assume therefore that $MN$ is a quartic extension of $K$ and
let $L$ denote the third quadratic subfield of $MN/K$.
The computations after Lemma \ref{L1} showed that 
$\mu^{1-\psi} = \alpha_\psi^2,\nu^{1-\psi} = \beta_\psi^2$,
and $\alpha_\psi^{1+\psi} = \beta_\psi^{1+\psi} = -1$
for $\psi \in  \{\rho, \sigma, \tau\}$. The proof of 
Lemma \ref{L1}
shows immediately that $\Gal(L/\Q) \simeq (2, 2, 2, 2)$,
and this proves our claim that $\nu \eq \delta\mu$ for some
$\delta \in \Z$. If $M/k$ and $N/k$ are unramified 
outside $\infty$, then so is $L/k$, and now the last
claim follows from Prop. \ref{P3}.
\end{proof}

Now let $M = K(\sqrt{\mu}\,)$ be the unramified $H_8$-extension
of $k$ constructed above (i.e. $\mu \in \Q(\sqrt{d_1}, \sqrt{d_2}\,)$
with $d_1d_2$ odd), and put $N = K(\sqrt{\nu}\,)$,
where $\nu \in \Q(\sqrt{d_1}, \sqrt{d_3}\,)$. By Lemma \ref{L3}
we know that over $K$ we have $\nu \eq2 \delta\mu$ for some
$\delta \in \Z$. Since $\mu$ is $2$-primary,
so is $m\mu$ or $-m\mu$ for every odd $m \in \Z$; we may
therefore assume without loss of generality that 
$\delta \in \{\pm 1, \pm 2\}$. 

If $4\nmid d_3$ and $\delta = \pm 1$, then we are done. 
If $d_3 \equiv 0 \bmod 8$, then $2$ or $-2$ is 
$2$-primary in $K$, and our claims also follow.
Finally, if $d_3 \equiv 4 \bmod 8$, then one
of $\{\nu, 2 \nu\}$ must be $2$-primary (note that
$-1$ is $2$-primary in this case).
This completes the proof of Prop. \ref{Pcon}.

\medskip

We have already remarked that the construction of the 
quaternion extension is much simpler when one of the
discriminants is a sum of two squares: assume for example
that $d_2 = s^2 + t^2$; then we can choose $a = 1$, and 
$z_1 = s$, $z_2 = 1$ and $z_3 = t$ 
solve equation (III), $y_1 = y_3 = 1$ and $y_2 = 0$ solve (II).
The following congruences help to choose the correct signs:
$$ \begin{array}{rcll}
\Big( 1+\sqrt{2a}\frac{1+\sqrt{b}}{2}\,\Big)^2  & \equiv & 
   (\sqrt{b}+\sqrt{2a}\,)(1+\sqrt{2a}\,) & \text{ if }
	a \equiv 1 \bmod 2, b\equiv 1 \bmod 8; \hfil \\
\Big( 1+\sqrt{2a}\frac{1-\sqrt{b}}{2}\,\Big)^2  & \equiv &
   (\sqrt{b}+\sqrt{2a}\,)(1-\sqrt{2a}\,) & \text{ if }
	 a \equiv 1 \bmod 2, b\equiv 5 \bmod 8; \hfil \\
(2\sqrt{a}+\sqrt{b}\,)(2+\sqrt{b}\,) & \equiv  & 1 \bmod 4
 	& \text{ if } a \equiv b \equiv 1 \bmod 4; \hfil \\
 (2\sqrt{a}-\sqrt{b}\,)\sqrt{b}\, & \equiv & (1+\sqrt{ab}\,)^2 \bmod 4
	& \text{ if } a \equiv 3 \bmod 4, b \equiv 1 \bmod 4. \\
\end{array}$$

\begin{table}[ht]
\caption{}
$$\begin{tabular}{|r|r|r|r|c|c|} \hline
 \rsp $d$  & $d_1$	& $d_2$ &  $d_3$ & $\mu$ & $\Cl(k)$ \\ \hline
 \rsp $3848$  & $8$  & $13$   & $37$   & 
	$(12\sqrt{2}+5\sqrt{13}\,)(18-5\sqrt{13}\,)$ & $(2,2)$\\
 \rsp $2120$ & $5$	& $8$	& $53$  &
	$(3\sqrt{5}+7\sqrt{2}\,)(1+\sqrt{2}\,)$ & $(2,2)$\\
 \rsp $1480$ & $5$	& $8$	& $37$  &
	$(3\sqrt{5}+2\sqrt{2}\,)(2-\sqrt{5}\,)$ & $(2,2)$\\
 \rsp $520$	& $5$	& $8$	& $13$  & 
	$(3\sqrt{2}+\sqrt{5}\,)(1+\sqrt{2}\,)$ & $(2,2)$\\
\hline
 \rsp $-120$ & $-3$ & $5$ & $8$ & 
	$(2\sqrt{2}+\sqrt{5}\,)(2+\sqrt{5}\,)$ & $(2,2)$ \\
 \rsp $ -255$ & $-3$   & $5$   & $17$ &
	$(\sqrt{5}+2\sqrt{-3}\,)(2+\sqrt{5}\,)$ & $(2,2,3)$ \\
 \rsp $-420$   & $-4$  & $5$  & $21$ &
	$(4i-\sqrt{5}\,)(2+\sqrt{5}\,)$ & $(2,2,2)$ \\
 \rsp $-455$  & $-7$	& $5$	& $13$	& 
	$(2\sqrt{13}-3\sqrt{5}\,)(2+\sqrt{5}\,)$ & $(2,2,5)$ \\
 \rsp $-520$	& $-8$	& $5$	& $13$	& 
	$(2\sqrt{-2}+\sqrt{5}\,)(2+\sqrt{5}\,)$ & $(2,2)$ \\
\hline\end{tabular}$$
\end{table}

See Table 2 for some numerical examples.
We will collect our main results in 

\begin{theorem}\label{thm1}
Let $k$ be a quadratic number field with discriminant $d$.
Then the following assertions are equivalent:
\begin{enumerate}
\item There exists an unramified $H_8$-extension $M/k$
	such that $M/\Q$ is normal;
\item There is a factorization $d = d_1d_2d_3$ of $d$ into
	three discriminants which are relatively prime
	and which satisfy the conditions
	$ (d_1d_2/p_3) = (d_2d_3/p_1) = (d_3d_1/p_2) = +1$ 
	for all $p_i\mid d_i$.
\end{enumerate}
\end{theorem}

Next we will show that this unramified $H_8$-extension
is unique if the discriminants $d_i$ are prime:

\begin{prop}\label{PUn}
Let $k$ be a quadratic number field with discriminant 
$d = d_1d_2d_3$ and assume that $M = K(\sqrt{\mu}\,)$ 
and $N$ are unramified $H_8$-extensions of $k$ with 
common subfield $K = k(\sqrt{d_1}, \sqrt{d_2}, \sqrt{d_3}\,)$; 
then there exists a discriminant $\delta \mid d$ such that
$N = K(\sqrt{\delta\mu}\,)$. 
\end{prop}

\begin{proof}
This is proved exactly as Lemma \ref{L3}: 
$MN$ contains a subfield $L$ such that $L/k$ is
unramified of type $(2,2,2)$ over $k$. Then Prop. \ref{P3}
shows that $L = K(\sqrt{\delta}\,)$ for some discriminant
$\delta$ dividing $d$. 
\end{proof}

\begin{cor}\label{Cuni}
Suppose that $d = d_1d_2d_3$ and that the $d_i$ are prime 
discriminants; then there exists at most one unramified
$H_8$-extension of $k$. 
\end{cor}

\begin{proof}
Since the $d_i$ are prime discriminants, we see
$\sqrt{\delta} \in K$; this implies $M = N$ in Prop. \ref{PUn}.
\end{proof}

\section{Ramification at $\infty$}

Assume that $d = d_1d_2d_3$ is an $H_8$-factorization.
If $d$ has a prime factor $q \equiv 3 \bmod 4$, then
there alwas exists an $H_8$-extension of $k$
containing $K = \Q(\sqrt{d_1}, \sqrt{d_2}, \sqrt{d_3}\,)$
which is unramified everywhere. In fact, the element $\mu$
constructed in Prop. \ref{Pcon} is either totally positive
or totally negative (since $K(\sqrt{\mu}\,)/\Q$ is normal);
if $\mu \gg 0$, $K(\sqrt{\mu}\,)/k$ is the sought extension,
and if $\mu \ll 0$, then we take $K(\sqrt{-q\mu}\,)/k$.

If, on the other hand, $d$ is the sum of two squares,
then either every unramified $H_8$-extension of $k$ is
also unramified at $\infty$, or none is. The following 
proposition tells us that this depends only on certain
biquadratic reciprocity symbols (the special case where all
the $d_i$ are prime discriminants such that $(d_i/d_j) = -1$ 
for $i \ne j$ is due to Hettkamp \cite{Het}):

\begin{prop}
Let $d_1, d_2, d_3$ be positive discriminants, none of which
is divisible by a prime $q \equiv 3 \bmod 4$,
and assume that $d = d_1d_2d_3$ is an $H_8$-factorization.
Then any $H_8$-extension $M/k$ of 
$k = \Q(\sqrt{d}\,)$ containing $\Q(\sqrt{d_1},\sqrt{d_2},\sqrt{d_3}\,)$
and unramified outside $\infty$ is totally real if and only if
$$ \Big(\frac{d_1d_2}{d_3}\Big)_4 
   \Big(\frac{d_2d_3}{d_1}\Big)_4 
   \Big(\frac{d_3d_1}{d_2}\Big)_4 \ = \ 
   \Big(\frac{d_1}{d_2}\Big)\Big(\frac{d_2}{d_3}\Big)
   \Big(\frac{d_3}{d_1}\Big). $$
Here $(d_1d_2/d_3)_4 = \prod_{p_3 \mid d_3} (d_1d_2/p_3)_4$,
and $(d/2)_4 = (-1)^{(d-1)/8}$.
\end{prop}

\begin{proof}
Let us first assume that $d \equiv 1 \bmod 4$. 
If $(x,y,z)$ is a solution of 
\begin{equation}\label{EE} d_1x^2 - d_2y^2 = -d_3z^2, \end{equation}
then $x$ is even and $yz$ is odd. Write $d_2 = a^2+b^2$
as a sum of squares with $2 \mid b$. Then the square
root of $\mu = (x\sqrt{d_1} + y \sqrt{d_2}\,)(b+\sqrt{d_2})$
generates the unramified  $H_8$-extension of $k$
if and only if $\mu$ is $2$-primary. But the congruences
$xb \equiv 0$, $b\sqrt{d_2} \equiv b$, $x \sqrt{d_1d_2} \equiv x
\bmod 4$ show that $\mu \equiv b+x+y \bmod 4$. 
Therefore, $\mu$ is $2$-primary if and only if 
$b+x+y \equiv 1 \bmod 4$. This can be achieved
by replacing $\mu$ by $-\mu$; assume therefore
that $\mu$ is $2$-primary.

Changing the sign of $b$ does not affect the primarity
of $\mu$. We are therefore allowed to assume that $b > 0$. 
Then $\mu$ is totally positive if and only if $y > 0$. 

Next we compute a few residue symbols from Eq. (\ref{EE}).
Since it implies the congruence $d_1x^2 \equiv d_2y^2 \bmod d_3$,
we find
$$ \Big(\frac{d_1d_2}{d_3}\Big)_4 = \Big(\frac{xy}{d_3}\Big)
   \Big(\frac{d_2}{d_3}\Big).$$
Here we have used that $d_3$ is not divisible by any prime 
$q \equiv 3 \bmod 4$. In a similar way we get
$$ \Big(\frac{d_2d_3}{d_1}\Big)_4 = \Big(\frac{yz}{d_1}\Big)
   \Big(\frac{d_3}{d_1}\Big) \quad \text{ and } \quad 
   \Big(\frac{d_3d_1}{d_2}\Big)_4 = \Big(\frac{-1}{d_2}\Big)_4
	\Big(\frac{zx}{d_2}\Big)\Big(\frac{d_3}{d_2}\Big) .$$
Multiplying these equation yields
$$ \Big(\frac{d_1d_2}{d_3}\Big)_4 
   \Big(\frac{d_2d_3}{d_1}\Big)_4 
   \Big(\frac{d_3d_1}{d_2}\Big)_4 \ = \ 
	\Big(\frac{xy}{d_3}\Big)\Big(\frac{d_2}{d_3}\Big)
	\Big(\frac{yz}{d_1}\Big)\Big(\frac{d_3}{d_1}\Big)
	\Big(\frac{-1}{d_2}\Big)_4 \Big(\frac{d_3}{d_2}\Big) 
	\Big(\frac{zx}{d_2}\Big).$$
Now $(d_1d_2/d_3) = 1$ by assumption; moreover 
$(-1/d_2)_4 = (-1)^{b/2}$. On the other hand, looking 
at (\ref{EE}) modulo $y$ and $z$ we get
$(d_1/y) = (-1/y)(d_3/y)$ and $(d_1/z) = (d_2/z)$.
In order to compute $(x/d_2)$ we write $x = 2^jx'$,
where $x'$ is the odd part of $x$. Then
$(x/d_2)= (2/d_2)^j (x'/d_2) = (2/d_2)^j (d_2/x')$ and
$(x/d_3)= (2/d_3)^j (x'/d_3) = (2/d_3)^j (d_3/x')$;
from Eq. (\ref{EE}) we get $(d_2/x') = (d_3/x')$. Moreover,
$j = 1 \iff x \equiv 2 \bmod 4 \iff d_2d_3 \equiv 5 \bmod 8$,
hence $(2/d_2)^j (2/d_3)^j = (2/d_2d_3) = (-1)^{x/2}$. 
Collecting everything gives
$$ \begin{array}{rcl}
   \ds \Big(\frac{d_1d_2}{d_3}\Big)_4 
   \Big(\frac{d_2d_3}{d_1}\Big)_4 
   \Big(\frac{d_3d_1}{d_2}\Big)_4 & = &
    \ds (-1)^{b/2} (-1)^{x/2}\Big(\frac{-1}{y}\Big)
		\Big(\frac{d_3}{d_2}\Big) \\
     & = & \ds (-1)^{(b+x+|y|-1)/2} \Big(\frac{d_3}{d_2}\Big) .
		\end{array} $$
Now the congruence $b+x+y \equiv 1 \bmod 4$ shows that 
$y > 0$ is equivalent to 
$$ \Big(\frac{d_1d_2}{d_3}\Big)_4 
   \Big(\frac{d_2d_3}{d_1}\Big)_4 
   \Big(\frac{d_3d_1}{d_2}\Big)_4 \ = \ \Big(\frac{d_3}{d_2}\Big).$$
Our claim follows since $(d_1/d_2)(d_3/d_1) = (d_2d_3/d_1) = +1$. 

If one of the prime discriminants is divisible $8$ then there 
are a few complications, but the very same proof shows 
that the result is valid also in this case. We may
assume without loss of generality that $8 \mid d_2$. We start
with the equation
\begin{equation}\label{EE2}
d_1x^2 - 2my^2 = -d_3z^2,
\end{equation}
(where $m \equiv 1 \bmod 4$) and put 
$\mu = (x\sqrt{d_1}+y\sqrt{2m}\,)(t+\sqrt{2m}\,)$,
where $2m = t^2+u^2$ for some $t \equiv 1 \bmod 4$; observe 
that $xyz \equiv 1 \bmod 2$. Then $\mu$ is $2$-primary if and 
only if $y \equiv (2/d_1) \bmod 4$. As above, we find
$$ \Big(\frac{8md_1}{d_3}\Big)_4 = \Big(\frac{xy}{d_3}\Big), \quad
   \Big(\frac{8md_3}{d_1}\Big)_4 = \Big(\frac{yz}{d_1}\Big), \quad
   \Big(\frac{d_3d_1}{8m}\Big)_4 = (-1)^{(d_1d_3-1)/8}.$$
Now $(x/d_3) = (d_3/x) = (d_2/x) = (2m/x)$, 
$(y/d_3) = (-1/y)(y/d_1)$  and $(z/d_1) = (2/z)$; a routine 
computation modulo $16$ shows that $(2/x)(2/z) = (d_3d_1/8m)_4$.
This gives
$$ \Big(\frac{8md_1}{d_3}\Big)_4 
   \Big(\frac{8md_3}{d_1}\Big)_4 
   \Big(\frac{d_3d_1}{8m}\Big)_4 \ = \ \Big(\frac{-1}{y}\Big)
	 \ = \ (-1)^{(|y|-1)/2}.$$
As in the case $d \equiv 1 \bmod 4$ above, this implies
that a $2$-primary $\mu$ is totally positive if and
only if $(8md_1/d_3)_4(8md_3/d_1)_4(d_1d_3/8m)_4 = (2/d_1)$;
since $(2/d_3)(d_3/2) = +1$, our claim follows.
\end{proof}

\section{Unramified Dihedral Extensions}

Of course, the very same methods allow us to treat unramified
dihedral extensions of quadratic number fields. In fact,
the proofs in this case are much simpler than those for
quaternion extensions and are left as an exercise
(for complete proofs in a more general situation, 
see \cite{Le2}). In fact, it is easy to see that
unramified extensions $M/k$ of quadratic number fields
$k$ which are normal over $\Q$ have Galois group $D_4 \times C_2$.
Using the decomposition and inertia groups one finds
that the existence of such an extension implies a
factorization $d = \disc k = d_1d_2\cdot d_3$ into three
(relatively prime) discriminants such that
$(d_1/p_2) = (d_2/p_1) = +1$ for all $p_j\mid d_j$ ($j  = 1, 2)$.
On the other hand, such a factorization implies the existence
of an unramified $C_4$-extension $L$ of $\Q(\sqrt{d_1d_2}\,)$,
and it is easy to see that the compositum $M = kL$ is an
unramified $D_4$-extension of $k$ such that  
$\Gal(M/\Q) \simeq D_4\times C_2$. We find

\begin{theorem}\label{Th}
Let $k$ be a quadratic number field with discriminant $d$ 
and $M/k$ an unramified $D_4$-extension such that $M/\Q$
is normal. Then $\Gal(M/\Q) \simeq D_4 \times C_2$, and there exists 
a "$D_4$-factorization" $d = d_1d_2 \cdot d_3$ into 
discriminants $d_i$ such that
\begin{enumerate}
\item[(i)] $(d_i,d_j) = (1)$ for $i \ne j$, and at most one of $d_1$ 
	or $d_2$ is negative; 
\item[(ii)] $(d_1/p_2) = (d_2/p_1) = +1$ for all primes
   $p_1 \mid d_1$ and $p_2 \mid d_2$. 
\end{enumerate} 
Moreover, $M/k(\sqrt{d_j}\,)$ is cyclic for $j=3$ and of 
type $(2,2)$ for $j = 1 , 2$. 

If, on the other hand, $k/\Q$ is a quadratic extension with 
discriminant $d = \disc k$, and if $d = d_1d_2 \cdot d_3$ is a 
$D_4$-factorization, then there is an $\alpha \in k(\sqrt{d_1}\,)$ 
such that $M = k(\sqrt{d_1},\,\sqrt{d_2},\,\sqrt{\alpha}\,)$ 
is a $D_4$-extension with the following properties:
\begin{enumerate}
\item $M/k$ is unramified outside $\infty$;
\item $M/k(\sqrt{d_3}\,)$ is cyclic; 
\item $M/\Q$ is normal with Galois group $D_4 \times C_2$. 
\end{enumerate}
\end{theorem}

\section*{Acknowledgement}
The author would like to thank R. Okazaki and K. Yamamura
for numerous suggestions, as well as the Laboratoire
d'Algorithmique Arithm\'etique eXp\'erimental at the
University of Bordeaux for their hospitality, and
the DFG, whose support has made my visit there possible.

\end{document}